\documentclass[a4paper,12pt]{amsart}
\usepackage[utf8]{inputenc}
\usepackage[T1]{fontenc}
\usepackage[english]{babel}
\usepackage{amsmath}
\usepackage{amsthm}
\usepackage{amssymb}
\usepackage{mathtools}
\usepackage{nicefrac}
\usepackage{hyperref}
\usepackage{adjustbox}
\usepackage{graphicx}
\usepackage{braket}
\usepackage{subfig}
\usepackage{csquotes}
\usepackage[nobysame]{amsrefs}
\providecommand{\abs}[1]{\left|#1\right|}
\providecommand{\norm}[1]{\left \| #1\right \|}

\addto\extrasenglish{%
}

\addto\extrasenglish{%
}

\theoremstyle{plain}
\newtheorem{theorem}{Theorem}[]

\theoremstyle{plain}
\newtheorem{lemma}{Lemma}[]

\theoremstyle{plain}
\newtheorem{corollary}{Corollary}[]

\theoremstyle{plain}

\theoremstyle{remark}
\newtheorem{remark}{Remark}[]

\theoremstyle{plain}
\newtheorem{definition}{Definition}[]

\title[Uniqueness results for higher order systems]{Uniqueness Results for Higher Order Elliptic Equations and Systems}

\author[D.~Cassani]{Daniele Cassani$^\text{1}$}
\author[D.~Schiera]{Delia Schiera$^\text{2}$}

\address[D. Cassani]{\newline\indent Dip. di Scienza e Alta Tecnologia
\newline\indent
Universit\`{a} degli Studi dell'Insubria
\newline\indent and
\newline\indent RISM--Riemann International School of Mathematics
\newline\indent Villa Toeplitz, Via G.B. Vico, 46 -- 21100 Varese}
\email{\href{mailto:Daniele.Cassani@uninsubria.it}{Daniele.Cassani@uninsubria.it}}

\address[D. Schiera]{\newline\indent Dip. di Scienza e Alta Tecnologia
\newline\indent
Universit\`{a} degli Studi dell'Insubria
\newline\indent via Valleggio 11 -- 22100 Como}
\email{\href{mailto:d.schiera@uninsubria.it}{d.schiera@uninsubria.it}}

\thanks{(1) Corresponding author: \texttt{daniele.cassani@uninsubria.it}}

\date{\today}

\begin{document}

\begin{abstract} In this paper we develop a Gidas-Ni-Nirenberg technique for polyharmonic equations and systems of Lane-Emden type. As far as we are concerned with Dirichlet boundary conditions, we prove uniqueness of solutions up to eighth order equations, namely which involve the fourth iteration of the Laplace operator. Then, we can extend the result to arbitrary polyharmonic operators of any order, provided some natural boundary conditions are satisfied but not for Dirichlet's: the obstruction is apparently a new phenomenon and seems due to some loss of information though far from being clear. When the polyharmonic operator turns out to be a power of the Laplacian, and this is the case of Navier's boundary conditions, as byproduct uniqueness of solutions holds in a fairly general context. New existence results for systems are also established.
\end{abstract}
\maketitle

%--------------------------------------------------------------------------------------------------------------------------
\section{Introduction}%------------------------------------------------------------------------------------------
%--------------------------------------------------------------------------------------------------------------------------

\noindent From the seminal paper of Gidas-Ni-Nirenberg \cite{GidasNiNirenberg79}, it is well known that the Lane--Emden equation
\begin{equation}\label{GNN}\begin{cases}
-\Delta u = \abs{u}^p\,, \,  &\text{ in } B_1 \subset \mathbb{R}^N, N > 2,\\
u=0\,, \, &\text{ on } \partial B_1
\end{cases} \end{equation}
with $1 < p < \frac{N+2}{N-2}$ has at most one, actually exactly one nontrivial solution, which is positive, radially symmetric and strictly decreasing in the radial variable ($B_1$ denotes the unit ball centered at the origin of $\mathbb{R}^N$). This result has been extended in many different directions and in particular to the biharmonic operator subject to Dirichlet boundary conditions
\begin{equation}\label{biheq} \begin{cases}
\Delta^2 u = \abs{u}^{p}\,, \,  &\text{ in } B_1 \subset \mathbb{R}^N, N > 4\\
u=\frac{\partial u}{\partial \nu}=0\,, \, &\text{ on } \partial B_1
\end{cases} \end{equation}
in \cite{Dalmasso95,FerreroGazzolaWeth07}, see also \cite{Dalmasso99} for the sublinear case, namely when $p<1$.
\noindent Uniqueness results have been also proved for the Lane--Emden system
\begin{equation}\label{hsys} \begin{cases}
\begin{aligned}
-\Delta u &= \abs{v}^{q}\,, \\
-\Delta v &= \abs{u}^{p}\,,
\end{aligned} & \text{ in } B_1 \subset \mathbb{R}^N, N > 2\\
u=v=0\,, & \text{ on } \partial B_1
\end{cases} \end{equation}
with $p, q>1$, see \cite{Dalmasso04} and then extended in  \cite{CuiWangShi07} to systems with more than two equations.

\noindent More recently, the non-variational situation has been addressed in \cite{Schiera18}, where uniqueness of solutions is established for the following system
\begin{equation}\label{deliasys} \begin{cases}
\begin{aligned}
\Delta^2 u=\abs{v}^q\,, \\
-\Delta v= \abs{u}^p\,,
\end{aligned} & \text{ in } B_1\subset \mathbb{R}^N, N > 4 \\
u=\frac{\partial u}{\partial \nu}=v=0\,, & \text{ on } \partial B_1\,.
\end{cases}\,
\end{equation}

%--------------------------------------------------------------------------------------------------------------------------
%---------------------------------------------------------------------------------------------
%--------------------------------------------------------------------------------------------------------------------------
\noindent In what follows $\Delta^\alpha(\cdot):=\Delta(\Delta^{\alpha-1}(\cdot))$, $\alpha\in\mathbb{N}$, $\alpha\geq 1$, denotes the iterated Laplace operator, the so-called polyharmonic operator. Clearly, this definition does not take into account boundary conditions according to which the iterated operator can be a power or not of the Laplacian. Higher order problems are more sensitive to boundary conditions with respect to the second order case, and   the richness of plenty of physically significant boundary values is the challenge which prevents to using standard tools from elliptic theory for second order operators, such as the maximum principle which fails in general for domains which are not slight perturbations of the ball; sufficient conditions for the validity of a general maximum principle have been addressed in \cite{CASTA}.
\subsection*{Main results} Before stating our main results let us make precise the notion of solution which in this context is always assumed to be in the classical pointwise sense. Our main results are the following:
\begin{theorem}\label{teo:eqnDir}
There exists at most one nontrivial solution to
\begin{equation}\label{eqnDir}
\begin{cases}
(-\Delta)^{\alpha} u=\abs{u}^p\,, \, &  \text{ in } B_1 \subset \mathbb{R}^N, N > 2\alpha\\
\frac{\partial^k u}{\partial \nu^k}=0, \, &  \text{ on } \partial B_1\,, \, k \le \alpha-1
\end{cases}
\end{equation}
with $p > 1$ and $1 \le \alpha \le 4$.
\end{theorem}
\noindent Notice that in the case $\alpha=1$, \eqref{eqnDir} reduces to the Lane--Emden equation \eqref{GNN} and for $\alpha=2$ to the biharmonic equation \eqref{biheq}. The boundary conditions in \eqref{eqnDir} are the higher order Dirichlet boundary conditions for which the polyharmonic operator fails to be the power of the Laplace operator. Those conditions are particularly relevant form the point of view of applications, see \cite{GazzolaGrunauSweers10}, as well as from the theoretical point of view, as they prevent to use reduction methods, such as decomposing the equation into a system of lower order equations. Actually the result of Theorem \ref{teo:eqnDir} is stronger, in the sense that yields uniqueness of solutions in the sharp range of existence $1<p<(N+2\alpha)/(N-2\alpha)$ as a consequence of \cite[Theorem 8]{pucci_serrin} and \cite[Theorems 7.17--7.18]{GazzolaGrunauSweers10}.

\noindent We then extend Theorem \ref{teo:eqnDir} to systems of polyharmonic equations as follows
\begin{theorem}\label{teo:sysDir}
There exists at most one nontrivial solution to
\begin{equation}\label{sysDir}
\begin{cases}
(-\Delta)^{\alpha_j} u_j=\abs{u_{j+1}}^{p_j},\, j =1, \dots, m-1  \\
&\, \text{ in } B_1,\\
(-\Delta)^{\alpha_m} u_m=\abs{u_{1}}^{p_m} \\
\frac{\partial^k u_j}{\partial \nu^k}=0, \, k=0, \dots, \alpha_j-1, \, j =1, \dots, m,   & \, \text{ on } \partial B_1, 
\end{cases}
\end{equation}
 with $p_j \ge 1$ for any $j$, $\prod_{j=1}^m p_j >1$, $N > 2 \max \{ \alpha_j \}_j$ and $1 \le \alpha_j\le 4$ for any $j =1, \dots, m$, where $m \ge 1$.
Moreover, let $\alpha_j \in \mathbb{N}$ and $p_j >1$ for any $j$. Then, there exists a classical nontrivial solution to \eqref{sysDir} if there exists $l \in \{ 1, \dots, m \}$ such that
\begin{equation}\label{SerrinMulti}
N+2 \sum_{k=1}^m \alpha_{k+l} \prod_{j=0}^{k-1} p_{j+l} - N \prod_{j=1}^m p_j \ge 0,
\end{equation}
where $p_{k+m}:=p_k$ and $\alpha_{k+m}:=\alpha_k$ for any $k=1, \dots, m$.
\end{theorem}
\noindent If $m=1$ then \eqref{sysDir} reduces to \eqref{eqnDir} with $\alpha=\alpha_1$, and \eqref{SerrinMulti} coincides with the Serrin exponent $\frac{N}{N-2\alpha}$. In the case $m=2$, $\alpha_1=\alpha$, $\alpha_2=\beta$, $p_1=q$, $p_2=p$, \eqref{SerrinMulti} reduce to the Serrin curves:
\[ 2\beta q + N + 2\alpha pq - N pq \ge 0, \quad 2\alpha p + N + 2\beta pq - N pq \ge 0. \]
\noindent Notice that for $m=2$, $\alpha_1=1$ and $\alpha_2=1$, \eqref{sysDir} reduces to \eqref{hsys} whereas when $m=2$, $\alpha_1=2$ and $\alpha_2=1$ we have \eqref{deliasys}.

\noindent As we are going to see, when trying to extend the proof of Theorem \ref{teo:eqnDir} to the case $\alpha \ge 5$, one has to face technical difficulties due to the fact that Dirichlet boundary conditions prescribe the behavior only of the first $\alpha-1$ derivatives of the solution, and no information apparently can be retained for higher order derivatives. However, in this context new boundary conditions show up in a natural fashion for which we have the following
\begin{corollary}\label{teo:Var}
There exists at most one nontrivial solution to
\begin{equation}\label{sys:Var}
\begin{cases}
(-\Delta)^{\alpha} u=\abs{u}^{p}, & \text{ in } B_1,   \\
\Delta^{2k} u=0,\, 2k \le \alpha-1,  &  \text{ on } \partial B_1 \\
\frac{\partial}{\partial \nu} \Delta^{2k} u= 0, \,  2k+1 \le \alpha-1,   &\text{ on } \partial B_1
\end{cases}
\end{equation}
with $N > 2  \alpha$, $p > 1$ and $\alpha \in \mathbb{N}$, $\alpha\geq 1$. 
%Moreover, there exists a nontrivial classical solution to \eqref{sys:Var} provided $1 < p < \frac{N+2\alpha}{N-2\alpha}$.
\end{corollary}

\noindent Boundary conditions considered in \eqref{sys:Var}, on one side from the mathematical point of view enable us to split the equation into a system of equations subject to Dirichlet boundary conditions, on the other side, the Physical constraint makes vanishing higher order momenta along the boundary.

\medskip

\noindent As far as we are concerned with the so-called Navier boundary conditions, for which the polyharmonic operator is actually a power of the Laplacian and classical reduction methods apply, we have as byproduct of the previous results the following
\begin{corollary}\label{teo:Navier}
There exists at most one nontrivial solution to
\begin{equation}\label{sys:Nav}
\begin{cases}
\begin{aligned}
&(-\Delta)^{\alpha_j} u_j=\abs{u_{j+1}}^{p_j},\, j =1, \dots, m-1,  \\
&&\\
&(-\Delta)^{\alpha_m} u_m=\abs{u_{1}}^{p_m},
\end{aligned}  \hspace{-.2cm}\text{ in } B_1 ,  \\
\\
\Delta^k u_j=0, \, k=0, \dots, \alpha_j-1, \, j =1, \dots, m  \text{ on } \partial B_1
\end{cases}
\end{equation}
with $p_j \ge 1$ for any $j$, $\prod_{j=1}^m p_j >1$, $\alpha_j \in \mathbb{N}$, $m \ge 1$ and $N > 2 \max \{ \alpha_j \}_j$.
\end{corollary}
\noindent We mention that here the case $m=3$ and $\alpha_j=1$ was covered in \cite{CuiWangShi07}. Nonexistence results above the critical curve for \eqref{sys:Nav}, in the variational case $m=2$, $\alpha_1=\alpha_2$ have been established in \cite{liu_guo_zhang}. Existence of solutions below the critical curve follows buying the line of \cite{clement_felmer_mitidieri}, as it has been detailed in \cite{Schiera19} where also the non-variational case is tackled.

%--------------------------------------------------------------------------------------------------------------------------
\section{Polyharmonic equations with Dirichlet boundary conditions: proof of Theorem \ref{teo:eqnDir}}\label{sec:eqn}%------------
%--------------------------------------------------------------------------------------------------------------------------

\noindent Let us first recall the following preliminary results:
\begin{lemma}[Theorem 5.7 in \cite{GazzolaGrunauSweers10}]\label{HopfPoly}
Let $u$ be a nontrivial solution to \eqref{eqnDir}. Then $u>0$ on $B_1$ and for every $x \in \partial B_1$ one has
\[ \begin{cases}
\Delta^{\alpha/2} u(x) >0, & \text{ for $\alpha$ even,}\\
-\frac{\partial}{\partial \nu} \Delta^{(\alpha-1)/2} u(x) >0, & \text{ for $\alpha$ odd.}
\end{cases} \]
\end{lemma}
\begin{lemma}[Theorem 7.1 in \cite{GazzolaGrunauSweers10}]\label{lem:rad}
Let $u$ be a nontrivial solution to \eqref{eqnDir}.
Then it is radially symmetric and strictly decreasing in the radial variable.
\end{lemma}
\noindent We next prove a key ingredient for what follows:
\begin{lemma}\label{lem:shape}
Let $u$ be a nontrivial solution to \eqref{eqnDir}. Then, $\Delta^s u(0)<0$ if $1 \le s < \alpha$ is odd, and in this case it is increasing until the first zero, $\Delta^s u(0)>0$ if $1 \le s<\alpha$ is even, and in this case it is decreasing up to the first zero.
Moreover, if $\alpha \ge 2$ is even, then the following properties hold:
\begin{itemize}
\item $\Delta^{\alpha-j} u$ has exactly $\alpha-j+1$ zeros (including the last one in $r=1$) and $\alpha-j$ critical points in $(0, 1)$ if $\alpha-1 \ge j \ge \alpha/2+1$, exactly $j$ zeros and $j-1$ critical points in $(0, 1)$ if $1 \le j \le \alpha/2$;
\item $\Delta^s u(1)=0$ if $s \le  \alpha/2 -1 $, $\Delta^s u(1) > 0$ if $s \ge \alpha/2$, and $(\Delta^s u)'(1)=0$ if $s \le \alpha/2-1$,  $(\Delta^s u)'(1)\ge 0$ if $s \ge  \alpha/2$.
\end{itemize}
\noindent If $\alpha \ge 3$ is odd, then we have:
\begin{itemize}
\item $\Delta^{\alpha-j} u$ has exactly $\alpha-j+1$ zeros (including the last one in $r=1$) and $\alpha-j$ critical points in $(0, 1)$ if $\alpha-1 \ge j \ge (\alpha+1)/2$, exactly $j$ zeros and $j-1$ critical points in $(0, 1)$ if $1 \le j \le (\alpha-1)/2$;
\item $\Delta^s u(1)=0$ if $s \le (\alpha-1)/2$, $\Delta^s u(1) < 0$ if $s \ge (\alpha+1)/2$, and $(\Delta^s u)'(1)=0$ if $s \le (\alpha-3)/2$,  $(\Delta^s u)'(1) \le 0$ if $s \ge (\alpha-1)/2$.
\end{itemize}
\noindent (See \autoref{fig:Delta}).
\end{lemma}
\begin{proof}

\begin{figure}
\centering
\includegraphics[width=0.8\textwidth]{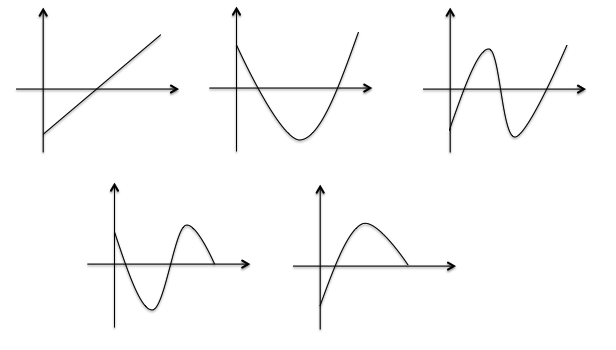}
\caption{\emph{Qualitative graphs of $\Delta^s u(r)$ on the interval $[0,1]$, where $s=5,4,3,2,1$ respectively, and $u$ satisfies \eqref{eqnDir} with $\alpha=6$.}}
\label{fig:Delta}
\end{figure}

\noindent We prove only the case in which $\alpha$ is even, the odd case being similar.  Recall that
\begin{equation}\label{radialLaplace} r^{N-1}(\Delta^{j} u)'(r)=\int_0^r s^{N-1} (\Delta^{j+1} u)(s) \, ds  \end{equation}
for any integer $j\geq 1$.
By \eqref{radialLaplace},  $(\Delta^{\alpha-1} u)'>0$ and as a consequence $\Delta^{\alpha-1} u$ has at most one zero. If $\alpha=2$, then in view of \autoref{HopfPoly} $\Delta u(1)>0$, hence $\Delta u$ has exactly one zero, and the proof is complete.
If $\alpha \ge 4$, then we conclude that $\Delta^{\alpha-2} u $ has at most two zeros. Indeed, again by \eqref{radialLaplace}, it is decreasing up to the endpoint $r_* \ge r_0$, where $r_0$ is such that $\Delta^{\alpha-1} u (r_0)=0$. Notice that if $r_* < 1$, then $(\Delta^{\alpha-2} u )'(r_*)=0$.
Therefore, it holds
\[ r^{N-1}(\Delta^{\alpha-2} u)'(r)=\int_{r_*}^r s^{N-1} (\Delta^{\alpha-1} u)(s) \, ds, \]
and since $\Delta^{\alpha-1} u >0$ beyond $r_* \ge r_0$, then $(\Delta^{\alpha-2} u)'(r) >0$ for any $r \ge r_*$.

\noindent Analogously, one concludes that $\Delta^{\alpha-j} u$ has at most $j$ zeros and $j-1$ critical points in $(0, 1)$, $j \le \alpha-1$.
In particular, $\Delta^{\alpha/2-1} u$ has at most $\alpha/2+1$ zeros and $\alpha/2$ critical points in $(0, 1)$.
Moreover, by Dirichlet boundary conditions, $\Delta^{\alpha/2-1} u(1)=0$, $(\Delta^{\alpha/2 -1} u )'(1)=0$
and $(\Delta^{\alpha/2 -1} u )''(1)=u^{(\alpha)}(1)=\Delta^{\alpha/2} u(1)> 0$ by \autoref{HopfPoly}.
Then, $\Delta^{\alpha/2-1} u$ should be decreasing and positive near 1.

\noindent Now, assume that $\Delta^{\alpha/2-1} u$ has exactly $\alpha/2+1$ zeros and $\alpha/2$ critical points in $(0, 1)$. Then $\Delta^{\alpha/2} u$ must have exactly $\alpha/2$ zeros, and by iteration $\Delta^{\alpha-j} u$ has exactly $j$ zeros, with $j \le \alpha/2+1$.
In particular, this means that $\Delta^{\alpha/2-1} u$ is positive near 0 and has a even number of zeros, if $\alpha/2 - 1$ is even; or it is negative near 0 and has a odd number of zeros, if $\alpha/2 - 1$ is odd. In any case, $\Delta^{\alpha/2-1} u$ should be increasing near $1$, a contradiction. Hence $\Delta^{\alpha/2-1} u$ must have one zero less, namely at most $\alpha/2$ zeros (including also the last one in $r=1$) and at most $\alpha/2 -1$ critical points in $(0, 1)$.

\noindent Now, let us consider $\Delta^{\alpha/2-2} u$. Since $\Delta^{\alpha/2-1} u$ has at most $\alpha/2$ zeros, of which the last one is in $r=1$, then it changes sing at most $\alpha/2$ times, and therefore $\Delta^{\alpha/2-2} u$ has at most $\alpha/2 - 1$ critical points, and $\alpha/2$ zeros in $(0,1)$.
Notice that $\Delta^{\alpha/2-2} u (1)=0$. Moreover, $(\Delta^{\alpha/2-2} u)^{(j)}(1)=0$ for any $j \le 3$ and $(\Delta^{\alpha/2-2} u)^{(4)}(1)=\Delta^{\alpha/2} u(1)> 0$. This means that $\Delta^{\alpha/2-2} u$ is decreasing and positive near 1. However, as above, this is possible only if $\Delta^{\alpha/2 -2} u$ has at most $\alpha/2-1$ zeros (including also the last one in $r=1$) and at most $\alpha/2 -2$ critical points.

\noindent Next we iterate the procedure. Then, at each step we lose one critical point.
Thus, $\Delta^{\alpha-j} u$ has at most $\alpha-j+1$ zeros (including the last one in $r=1$) and $\alpha-j$ critical points in $(0, 1)$ if $j \ge \alpha/2+1$,  at most $j$ zeros and $j-1$ critical points in $(0, 1)$ if $j \le \alpha/2$.
In particular, $\Delta u$ has at most 1 critical point. We know that $\Delta u (0) = u''(0) <0$, as $u'(0)=0$ and $u' <0$ in $(0, 1)$.
We have two cases: $\Delta u$ is increasing and negative, reaches a positive maximum and decreases to 0, or it is always negative and has no critical points.
However, we know that $\Delta u(1)=0$ and $\Delta u$ is decreasing in the last interval, as $(\Delta u)^{(j)}(1)=0$ for any $j \le \alpha -3$ and $(\Delta u)^{(\alpha -2)}(1) = u^{(\alpha)}(1)=\Delta^{\alpha/2} u(1) >0$ by \autoref{HopfPoly}.
Then necessarily $\Delta u$ is increasing and negative, reaches a positive maximum and decreases to 0, namely has exactly one critical point.

\noindent As a consequence, $\Delta^2 u$ has at least 2 critical points, however since it has at most 2 critical points due to what proved above, it turns out to have exactly 2 critical points.
Moreover, $\Delta^2 u(0)>0$, and it is decreasing until the first zero.

\noindent Iteratively, we conclude that $\Delta^{\alpha-j} u$ has exactly $\alpha-j+1$ zeros (including the last one in $r=1$) and $\alpha-j$ critical points in $(0, 1)$ if $j \ge \alpha/2+1$, exactly $j$ zeros and $j-1$ critical points in $(0, 1)$ if $j \le \alpha/2$.
Moreover, $\Delta^s u(0)<0$ if $s$ is odd, and in this case it is increasing until the first zero, $>0$ if $s$ is even, and in this case it is decreasing before the first zero.
Further, by boundary conditions, $\Delta^s u(1)=0$ if $s \le  \alpha/2 -1 $, $\Delta^s u(1) > 0$ if $s \ge \alpha/2$, and $(\Delta^s u)'(1)=0$ if $s \le \alpha/2-1$,  $(\Delta^s u)'(1)\ge 0$ if $s \ge  \alpha/2$.
\end{proof}

%--------------------------------------------------------------------------------------------------------------------------
\subsection{Proof of \autoref{teo:eqnDir} in the case $\alpha=3$}
%-------------------------------------------
\label{sec:3}%----------------------------------------------------------------------------------------------------------
%--------------------------------------------------------------------------------------------------------------------------

Let $u$ be a nontrivial solution to
\begin{equation}\label{eqn3}
\begin{cases}
-\Delta^3 u=\abs{u}^p, &  \text{ in } B_1\\
u=\frac{\partial u}{\partial \nu}=\frac{\partial^2 u}{\partial \nu^2}=0, & \text{ on } \partial B_1.
\end{cases}
\end{equation}
By \autoref{lem:rad} and \autoref{HopfPoly}, $u$ is positive, radially symmetric and strictly decreasing.
In particular, since the maximum is attained at $0$, we have $u'(0)=0$.
Moreover,
\[
r^{N-1}(\Delta^2 u)'(r)=\int_0^r s^{N-1} (\Delta^3 u)(s) \, ds .
\]
As a consequence,
\begin{equation}\label{ce}
(\Delta^2 u)'(0)=\lim_{r \to 0} \frac{\int_0^r s^{N-1} (\Delta^3 u)(s) \, ds}{r^{N-1}}=0.
\end{equation}
Moreover,
\[
r^{N-1}(\Delta u)'(r)=\int_0^r s^{N-1} (\Delta^2 u)(s) \, ds
\]
and therefore
\begin{equation}\label{ce2}
(\Delta u)'(0)=\lim_{r \to 0} \frac{\int_0^r s^{N-1} (\Delta^2 u)(s) \, ds}{r^{N-1}}=0.
\end{equation}

\noindent Let $w$ be another nontrivial solution to \eqref{eqn3} and set
\[ \tilde{w}(r)=\lambda^s w(\lambda r) , \]
 where $s$ is chosen such that $\tilde{w}$ satisfies
 \[ \begin{cases}
 -\Delta^3 \tilde w=\abs{\tilde w}^p, \, r \le 1/\lambda\\
\tilde w(1/\lambda)= \tilde w'(1/\lambda)= \tilde w''(1/\lambda)=0
\end{cases} \]
namely $s=\frac{6}{p-1}$,
whereas $\lambda >0$ is such that
 \begin{equation}\label{lambda}
\tilde{w}(0)= u(0).
 \end{equation}

\noindent \textbf{Claim}:
\begin{equation}\label{claim}
\Delta \tilde w(0)=\Delta u(0), \, \Delta^2 \tilde w(0)=\Delta^2 u(0).
\end{equation}

\noindent Let us suppose for instance $\Delta^2( u- \tilde w)(0)>0$ and $\Delta( u- \tilde w)(0) >0$. Notice that by continuity $\Delta^2( u- \tilde w) >0$ on $[0, \delta)$ and $\Delta( u- \tilde w) >0$ on $[0, \varepsilon)$ for some $\delta, \varepsilon$ sufficiently small. Moreover $u- \tilde w >0$ on $(0, \varepsilon]$: indeed, if there exists $a \le \varepsilon$ such that $u(a)-\tilde w(a) \le 0$, then $\Delta(u-\tilde w)>0$ implies $u- \tilde w <0$ on $[0, a)$, which is a contradiction.

\noindent Hence we can choose $R_1$ such that
\begin{multline*} R_1= \sup \{ r \le \min \{1, 1/\lambda \} : \, (u- \tilde w)(s) >0, \, \Delta( u- \tilde w)(s) >0, \\
 \, \Delta^2(u-\tilde w)(s) >0, \, s \in (0, r) \} .\end{multline*}
We have
\begin{equation}\label{R_1}
(u- \tilde w)(R_1) >0, \, \Delta(u-\tilde w)(R_1)>0.
\end{equation}
Indeed, let us assume by contradiction that $(u-\tilde w)(R_1)=0$. Then, since $\Delta(u-\tilde w) >0$ on $[0, R_1)$ we would have by the maximum principle $u-\tilde w <0$ on $[0, R_1)$.
Analogously, if $\Delta(u-\tilde w)(R_1)=0$, then $\Delta(u-\tilde w) <0$ on $(0, R_1)$, a contradiction.
As a consequence, \eqref{R_1} holds. Moreover, either $R_1<\min \{1, 1/\lambda \}$, and in this case $\Delta^2( u- \tilde w)(R_1)=0$, or $R_1=\min \{1, 1/\lambda \}$.

\noindent In the first case, by applying the maximum principle to $-\Delta^3(u- \tilde w) =u^p- \tilde w^p>0$, one has $\Delta^2( u- \tilde w) <0$ on $(R_1, R_1+ \delta)$ for $\delta$ sufficiently small.
We can set $R_2$ such that
\begin{multline*} R_2= \sup \{ r \le \min \{1, 1/\lambda \} : \, (u- \tilde w)(s) >0, \, \Delta( u- \tilde w)(s) >0, \,  \\
\Delta^2(u-\tilde w)(s) <0, \, s \in (R_1, r) \} .\end{multline*}
As above, we have
\[ \Delta^2( u- \tilde w)(R_2) <0, \, (u-\tilde w)(R_2)>0 \]
and either $R_2<\min \{1, 1/\lambda \}$, which implies $\Delta(u-\tilde w)(R_2)=0$, or $R_2=\min \{1, 1/\lambda \}$. Indeed, if $\Delta^2( u- \tilde w)(R_2)=0$, then by applying the maximum principle to $-\Delta^3(u - \tilde w) =u^p- \tilde w^p>0$ on $B_{R_2} \setminus \overline{B_{R_1}}$ we have $\Delta^2( u- \tilde w) >0$ on $(R_1, R_2)$; on the other hand, if $(u- \tilde w)(R_2)=0$, then $u- \tilde w <0$ on $[0, R_2)$, as $\Delta(u- \tilde w)>0$.

\noindent We now apply iteratively the same reasoning as above to get a sequence (which can be finite or infinite)
\[ 0 =R_0 < R_1 < R_2 < \dots \le \min \{1, 1/\lambda \} \]
such that
\[ u(R_{3k})=\tilde w(R_{3k}), \, \Delta^2 u(R_{3k+1})=\Delta^2 \tilde w(R_{3k+1}), \, \Delta u(R_{3k+2})=\Delta \tilde w(R_{3k+2}), \]
$k \ge 0$, as long as $R_k < \min \{1, 1/\lambda \} $, see \autoref{table3}.

\begin{table}\centering\caption{Sign of $u- \tilde w$, $\Delta(u- \tilde w)$, $\Delta^2( u- \tilde w)$. }
\label{table3}\begin{adjustbox}{max width=\textwidth}
\begin{tabular}{r|cccc}
&$(u-\tilde w)(s)$&$\Delta(u- \tilde w)(s)$&$\Delta^2( u- \tilde w)(s)$\\ \hline
$s=0$&=0&>0&>0\\ \hline
$s \in (0, R_1)$&>0&>0&>0\\ \hline
$s=R_1$&>0&>0&=0\\ \hline
$s \in (R_1, R_2)$&>0&>0&<0\\ \hline
$s=R_2$&>0&=0&<0\\ \hline
$s \in (R_2, R_3)$&>0&<0&<0\\ \hline
$s=R_3$&=0&<0&<0\\ \hline
$s \in (R_3, R_4)$&<0&<0&<0\\ \hline
$\vdots$ &$\vdots$ &$\vdots$ &$\vdots$
\end{tabular}
\end{adjustbox}
\end{table}

\noindent If it is infinite, then we take the limit $R_*=\lim_{i \to \infty} R_i \le  \min \{ 1, 1/\lambda \}$ and by continuity and differentiability, it holds
\[ (u- \tilde w )(R_*)=0, \, \Delta( u- \tilde w)(R_*)=0, \, \Delta^2(u- \tilde w)(R_*)=0 \]
and
\[ (u'- \tilde w' )(R_*)=0, \, (\Delta( u- \tilde w))'(R_*)=0, \, (\Delta^2(u- \tilde w))'(R_*)=0. \]
Now, one defines
 \[ U(r)=(u(r), -\Delta u(r), \Delta^2 u(r)) \quad 0 \le r \le 1 \]
 and
\[  W(r)=(\tilde{w}(r), -\Delta \tilde{w} (r), \Delta^2 \tilde w(r)) \quad 0 \le r \le 1/\lambda. \]
 Hence,  for any $0 \le r \le R_*$ one has
 \begin{multline}\label{UW_1} U(r)- W(r)=\\
 \int_{r}^{R_*} \frac{s}{N-2} \left( 1 - \left(\frac{s}{r} \right)^{N-2} \right)  (F(U(s))- F(W(s))) \, ds \end{multline}
where we set $F(x,y, z)=(y, z, x^p)$.
Since $p>1$, then $F$ is locally Lipschitz continuous, hence by the Gronwall Lemma, \eqref{UW_1} implies $U=W$ on $[0, R_*]$.
This is in contradiction with the assumption $\Delta^2( u- \tilde w)(0)>0$.

\noindent On the other hand, if the sequence stops at a maximum value $R_k$ then on $(R_{k-1}, R_k=\min \{1, 1/\lambda \} ]$ one of the following is verified, see \autoref{table3}:
\begin{itemize}
\item $u- \tilde w$ and $\Delta(u- \tilde w)$ have the same sign
\item $u- \tilde w$ and $\Delta^2(u- \tilde w)$ have opposite sign.
\end{itemize}

\noindent Let for instance $u- \tilde w >0$ and $\Delta(u- \tilde{w}) \ge 0 $. Then,
\[ 0 < (u- \tilde w)(\min \{1, 1/\lambda \})= \begin{cases}
u(1/\lambda) & \text{if $\lambda >1$}\\
0 & \text{if $\lambda=1$}\\
- \tilde w (1) & \text{if $\lambda <1$}
\end{cases} \]
which implies $\lambda >1$, whereas by Hopf lemma $$0< (u'- \tilde w' )(\min \{1, 1/\lambda \})=(u' - \tilde w')(1/\lambda)=u'(1/\lambda)<0$$ thus a contradiction.

\noindent Let now $u- \tilde w \ge 0$, $\Delta( u- \tilde w) < 0$ and $\Delta^2(u- \tilde w) <0$. Hence $(u - \tilde w)(\min \{1, 1/\lambda \}) \ge 0$, and therefore $\lambda \ge 1$. Moreover, $\Delta u(1/\lambda)=\Delta( u- \tilde w)(1/\lambda) < 0$, whereas by Hopf Lemma and \autoref{lem:shape} $(\Delta u)'(1/\lambda) \le (\Delta( u- \tilde w))'(1/\lambda) < 0$.

\noindent By \autoref{lem:shape}, in particular we have that $\Delta u$ increases until reaches a point $r_0$ and then decreases. Since $\Delta u(1)=0$, $\Delta u$ attains its maximum in $r_0$ and $(\Delta u)'<0$, $\Delta u >0$ on $(r_0, 1)$, whereas $(\Delta u)' >0$ on $(0,r_0)$.
Therefore, we cannot find a point such that $(\Delta u)'<0$ and $\Delta u  <0$, hence we reach again a contradiction.

\noindent Since we get to a contradiction in all possible cases, we can not have $\Delta^2( u- \tilde w)(0)>0$ and $\Delta( u- \tilde w)(0) >0$. In a similar fashion, one proves that also the other possible choices for the sign of $\Delta^2( u- \tilde w)(0)$ and $\Delta( u- \tilde w)(0)$ yield a contradiction, hence the claim \eqref{claim} holds.

\noindent Now, in view of \eqref{lambda} and \eqref{claim}, and since by \eqref{ce} and \eqref{ce2}
\begin{multline} u'(0)=\tilde w'(0)= (\Delta^2 u)'(0)= (\Delta^2 \tilde w)'(0)\\
=\tilde z'(0)=(\Delta u)'(0)=(\Delta \tilde w)'(0)=0,  \end{multline}
for any $r \le \min \{ 1, 1/\lambda \}$ one has
 \begin{multline}\label{UW} U(r)- W(r)=\\
 \int_0^r \frac{s}{N-2} \left( 1 - \left(\frac{s}{r} \right)^{N-2} \right)  (F(W(s))- F(U(s))) \, ds \end{multline}
 where $F(x,y, z)=(y, z, x^p)$.
Since $p>1$, then $F$ is locally Lipschitz continuous, hence by the Gronwall Lemma, \eqref{UW} implies $U=W$ on $[0, \min \{1, 1/\lambda\}]$.

\noindent Finally, $0<u(1/\lambda)=\tilde w(1/\lambda)=0$ if $\lambda >1 $, whereas $0=u(1)=\tilde w(1)>0$ if $\lambda <1 $, thus $\lambda=1$ and $u=w$. \qed

%--------------------------------------------------------------------------------------------------------------------------
\subsection{Proof of \autoref{teo:eqnDir} in the case $\alpha=4$}%-------------------------------------------
%--------------------------------------------------------------------------------------------------------------------------

Let $u, w$ be two nontrivial solutions to
\begin{equation}\label{eqn4}
\begin{cases}
\Delta^4 u=\abs{u}^p, & \text{ in } B_1\\
u=\frac{\partial u}{\partial \nu}=\frac{\partial^2 u}{\partial \nu^2}=\frac{\partial^3 u}{\partial \nu^3}=0, &  \text{ on } \partial B_1.
\end{cases}
\end{equation}
Choose $\lambda, s$ such that $\tilde w(r)=\lambda^s w(\lambda r)$ satisfies \eqref{eqn4} on $B_{1/\lambda}$ and $u(0)=\tilde w(0)$. We want to prove that
\begin{equation}\label{claim4}
\Delta^ku(0)=\Delta^k \tilde w (0), \quad k=0, \dots, 3.
\end{equation}
For instance, assume that
\[ \Delta(u- \tilde w)(0)>0, \quad \Delta^2(u- \tilde w)(0)<0, \quad \Delta^3(u- \tilde w)(0)>0.  \]
Considerations below hold with some modifications also for other choices of the above signs.
Let us define 
\begin{multline*} R_1= \sup \{ r \le \min \{1, 1/\lambda \} : \, (u- \tilde w)(s) >0, \, \Delta( u- \tilde w)(s) >0, \\
 \, \Delta^2(u-\tilde w)(s) < 0, \, \Delta^3(u-\tilde w)(s) >0,  s \in (0, r) \} .\end{multline*}
By the maximum principle, $(u- \tilde w)(R_1)>0$ and $\Delta^3(u- \tilde w)(R_1)>0$, whereas $\Delta(u-\tilde w)(R_1)$ and $\Delta^2(u- \tilde w)(R_1)$ may be $=0$.
If for instance $\Delta(u-\tilde w)(R_1)=0$, then by considering
\begin{multline*} R_2= \sup \{ r \le \min \{1, 1/\lambda \} : \,  (u-\tilde w)(s)>0, \, \Delta(u-\tilde w)(s) < 0, \\
\Delta^2(u- \tilde w)(s) <0, \, \Delta^3(u- \tilde w)(s) >0, \, s \in (R_1, r) \} \end{multline*}
we have that $\Delta(u-\tilde w)(R_2)<0$ and $\Delta^3(u- \tilde w)(R_2)>0$, whereas $(u-\tilde w)(R_2)$ and $\Delta^2(u- \tilde w)(R_2)$ may be $=0$. We now iterate to get a sequence $\{ R_j \}$ (finite or infinite) such that for any $j$ one or two among $(u- \tilde w)(R_j), \Delta(u-\tilde w)(R_j), \Delta^2(u- \tilde w)(R_j), \Delta^3(u- \tilde w)(R_j)$ is $=0$.

\noindent If $\{ R_j \}$ is infinite, then we reach a contradiction as in \autoref{sec:3} by applying the Gronwall Lemma with $F(x, y, z, w)=(y, z, w, x^p)$. Let us assume that $\{ R_j \}$ is finite.
We want to exclude the possibility that on $(R_j, R_{j+1})$ for some $j$ we have
\[ (u- \tilde w)<0, \, \Delta(u-\tilde w)>0, \,  \Delta^2(u-\tilde w)<0, \,  \Delta^3(u-\tilde w)>0 \]
(or opposite signs).
In order for this to happen, since in $(0, R_1)$
\[ (u- \tilde w)>0, \, \Delta(u-\tilde w)>0, \,  \Delta^2(u-\tilde w)<0, \,  \Delta^3(u-\tilde w)>0 \]
we need that  $(u-\tilde w)(R_k)=0$ for an odd number of $k \le j$, $\Delta(u-\tilde w)(R_k)=0$ for an even number of $k \le j$, $ \Delta^2(u-\tilde w)(R_k)=0$ for an even number of $k \le j$, and $ \Delta^3(u-\tilde w)(R_k)=0$ for an even number of $k \le j$.
However, let us assume that the number of $k\le j$ such that $(u-\tilde w)(R_k)=0$ is $n$. Then, the number of zeros of $\Delta(u-\tilde w)$ must be $\ge n$, since $u-\tilde w$ can be $0$ only if $\Delta(u-\tilde w)$ has been $=0$ before.
There are three possible cases:
\begin{enumerate}
\item The number of zeros of $\Delta(u-\tilde w)$ is $n$;
\item The number of zeros of $\Delta(u-\tilde w)$ is $n+1$ (if we stop after a zero of $\Delta(u-\tilde w)$ and before $(u-\tilde w)$ vanishes again);
\item The number of zeros of $\Delta(u-\tilde w)$ is equal to $n+2$. This last case happens when $\Delta(u-\tilde w)=0$ for two consecutive times, without having $(u-\tilde w)=0$ in the between.
Notice that such a situation may happen just once, since at the last step the four columns turn out to have the same sign and hence cannot be $0$ again, see a model case in \autoref{table4}.
\end{enumerate}
\begin{table}\centering\caption{Sign of $u- \tilde w$, $\Delta(u- \tilde w)$, $\Delta^2(u- \tilde w)$, $\Delta^3(u- \tilde w)$ in a special case. }
\label{table4}\begin{adjustbox}{max width=\textwidth}
\begin{tabular}{r|ccccc}
&$(u-\tilde w)(s)$&$\Delta( u- \tilde w)(s)$&$ \Delta^2(u-\tilde w)(s)$&$ \Delta^3(u-\tilde w)(s)$\\ \hline
$\vdots$ &$\vdots$ &$\vdots$ &$\vdots$ &$\vdots$\\ \hline
$s \in (R_{j}, R_{j+1})$&<0&<0&>0&>0\\ \hline
$s=R_{j+1}$&<0&=0&>0&>0\\ \hline
$s \in (R_{j+1}, R_{j+2})$&<0&>0&>0&>0\\ \hline
$s=R_{j+2}$&<0&>0&>0&=0\\ \hline
$s \in (R_{j+2}, R_{j+3})$&<0&>0&>0&<0\\ \hline
$s=R_{j+3}$&<0&>0&=0&<0\\ \hline
$s \in (R_{j+3}, R_{j+4})$&<0&>0&<0&<0\\ \hline
$s=R_{j+4}$&<0&=0&<0&<0\\ \hline
$s \in (R_{j+4}, \min \{1, 1/\lambda \})$&<0&<0&<0&<0\\ \hline
\end{tabular}
\end{adjustbox}
\end{table}

\noindent Assume $n$ odd. In order to have an even number of zeros of $\Delta(u-\tilde w)$ we have to consider the second case, namely the number of zeros of $\Delta(u-\tilde w)$ must be $n+1$.
Now, $\Delta(u-\tilde w)$ might be zero in $R_1$ even if $ \Delta^2(u-\tilde w)$ has not vanished yet. Hence the number of zeros of $ \Delta^2(u-\tilde w)$ can be $n$, $n+1$ or $n+2$. Recall we need that $\Delta^2(u- \tilde w)$ has an even number of zeros, and that $n$ is odd, hence we conclude that $\Delta^2(u- \tilde w)$ has $n+1$ zeros. We deduce as above that $ \Delta^3(u-\tilde w)$ can have $n$, $n+1$ or $n+2$ zeros, and in turn $n+1$ since their number has to be even. However, this implies that $u-\tilde w$ should have at least $n+1$ zeros. This is a contradiction, since the number of zeros of  $u-\tilde w$ is $n$ by assumption.

\noindent As a consequence, we conclude that  the following configuration is not possible 
\[ (u- \tilde w)<0, \, \Delta(u-\tilde w)>0, \,  \Delta^2(u-\tilde w)<0, \,  \Delta^3(u-\tilde w)>0, \]
and the same holds true having opposite signs.
Therefore, one of the following (or reversed) is verified on $(R_k, \min \{1, 1/\lambda \})$:
\begin{itemize}
\item $(u- \tilde w)>0$, $\Delta(u-\tilde w)>0$;
\item $(u- \tilde w)>0$, $\Delta(u-\tilde w)<0$, $\Delta^2(u-\tilde w)<0$;
\item $(u- \tilde w)>0$, $\Delta(u-\tilde w)<0$, $\Delta^2(u-\tilde w)>0$, $\Delta^3(u-\tilde w)>0$.
\end{itemize}

\noindent By \autoref{lem:shape}, $\Delta^3 u$ is increasing. Moreover, $\Delta^3 u(0) <0$, and $\Delta^2 u$ is first positive and decreasing, then negative, reaches its minimum in this interval and then increases to a positive value $\Delta^2 u(1)$. As a consequence, $\Delta u(0)<0$, then increases, reaches a positive maximum value and then decreases to 0.

\noindent Assume that in the last interval the following holds 
\[ (u- \tilde w)>0, \, \Delta(u-\tilde w)>0.  \]
If both the first and the second column have $n$ zeros, then we apply the Hopf lemma and we obtain $0>u'(1/\lambda)=(u'-\tilde w')(1/\lambda)>0$, a contradiction. Otherwise, it means that the second column has $n+2$ zeros, which in turn gives that the third column has $n+1$ zeros, and the last one has $n$ zeros, thus $ \Delta^2(u-\tilde w)>0$ and $\Delta^3(u-\tilde w)>0$, see \autoref{table4}. Then, by applying Hopf lemma,
\[ 0<(\Delta^2 (u- \tilde w))'(1/\lambda)\le (\Delta^2 u)'(1/\lambda) \]
as $(\Delta^2 \tilde w)'(1/\lambda)\ge 0$, and $0<\Delta^2 (u-\tilde w)(1/\lambda)< \Delta^2 u(1/\lambda)$. Moreover, $\Delta u(1/\lambda)=\Delta(u- \tilde w)(1/\lambda)>0$ and $(\Delta u)'(1/\lambda)=(\Delta (u-\tilde w))'(1/\lambda)>0$. However, by \autoref{lem:shape}, there does not exist a point such that $\Delta^2 u>0$, $(\Delta^2 u)'>0$, $\Delta u>0$ and $(\Delta u)'>0$.

\noindent Assume that
\[ (u- \tilde w)>0, \, \Delta(u-\tilde w)<0, \, \Delta^2(u-\tilde w)<0.  \]
Then $\lambda >1$, $0>\Delta(u-\tilde w)(1/\lambda)=\Delta u(1/\lambda)$. If $\Delta^2(u-\tilde w)$ does not change sign after the last zero of $\Delta(u-\tilde w)$, then we can apply Hopf to get $(\Delta u)'(1/\lambda)=(\Delta(u-\tilde w))'(1/\lambda)<0$.
However, it cannot exists a point such that $\Delta u(1/\lambda)<0$ and $(\Delta u)'(1/\lambda)<0$ by \autoref{lem:shape}.
If we cannot apply Hopf, then it means that the third column has $n+2$ zeros, which is not possible.

\noindent Assume finally that
\[ (u- \tilde w)>0, \, \Delta(u-\tilde w)<0, \, \Delta^2(u-\tilde w)>0, \, \Delta^3(u-\tilde w)>0, \]
Again $\lambda >1$ and $0>\Delta u(1/\lambda)$. Moreover, $0<\Delta^2 (u-\tilde w)(1/\lambda)< \Delta^2 u(1/\lambda)$ and by Hopf
\[ 0<(\Delta^2 (u- \tilde w))'(1/\lambda)\le (\Delta^2 u)'(1/\lambda) \]
as $(\Delta^2 \tilde w)'(1/\lambda) \ge 0$ by \autoref{lem:shape}.
However, such a point cannot exists, hence we have a contradiction.
As in Section \ref{sec:3}, we conclude that \eqref{claim4} holds, then $u=\tilde w$, which in turn gives $u=w$. \qed

\medskip

\subsection*{Open problem} Consider $\alpha \ge 5$, and take two different solutions $u, w$. One can naturally parametrize $w$ as $\tilde w(r)=\lambda^s w(\lambda r)$, where $s=\frac{2\alpha}{p-1}$, and $\lambda$ is such that $\tilde w(0)=u(0)$. Again, it is easy to prove that the uniqueness result follows once we prove that $\Delta^k(u- \tilde w)(0)=0$. One builds a table as above, and gets a sequence $\{ R_j \}$. If it is infinite, then one extends considerations above choosing a suitable $F$ to apply Gronwall. The main difficulty turns out to be the proof of the contradiction in the finite case, equivalently, the extension of the following lemma to $\alpha \ge 5$.
\begin{lemma}
Let $2 \le \alpha \le 4$. 
Then the following configuration:  
\[ (-\Delta)^k (u-\tilde w)<0, \, k=0, \dots, \bar{k} \]
and
\[ (-\Delta)^{\bar{k}+1} (u-\tilde w)>0, \]
for some $\bar{k}$, cannot occur at the last step. 
\end{lemma}

As a consequence we have 
\begin{lemma}\label{lem:zeros}
Let $2 \le \alpha \le 4$. Assume that $u-\tilde w$ has $n$ zeros and that $u(0)=\tilde w(0)$ and
\[ (-\Delta)^k (u-\tilde w) (0)<0, \, k=1, \dots, \alpha-1 \]
holds. Then $\Delta^{\alpha-1}(u-\tilde w)$ must have at least $n+1$ zeros.
\end{lemma}
Indeed, if not, then at least two consecutive columns have the same sign, and we get a contradiction. 

\begin{remark}\label{rmk:zeros}
One can prove in the same way as \autoref{lem:zeros} that, if $\alpha \le 4$ and
\[ (-\Delta)^k (u-\tilde w)(0) <0, \, k=0, \dots, \alpha-1 \]
holds, and $u-\tilde w$ has $n$ zeros, then $\Delta^{\alpha-1}(u-\tilde w)$ must have at least $n$ zeros. This will be useful in the next Section.
\end{remark}

%--------------------------------------------------------------------------------------------------------------------------
\section{Proof of Theorem \ref{teo:sysDir}}%--------------------------------------------------------------------------------------
%--------------------------------------------------------------------------------------------------------------------------

%--------------------------------------------------------------------------------------------------------------------------
\subsection{Existence}%--------------------------------------------------------------------------------------------
%--------------------------------------------------------------------------------------------------------------------------

Next we extend to system \eqref{sys:Var} the existence results obtained in \cite{Schiera18} in the case of systems of two equations, see also \cite{AziziehClementMitidieri02} for $p$-Laplacian systems. In what follows we recall the main steps in the proof, and the necessary changes required to treat the case in which one has $m >2$ equations.
\medskip

\noindent \textit{Step 1}. An auxiliary system. If $\prod_{j=1}^m p_j>1$, and if the only classical solution to \eqref{sysDir} is the trivial one, then there exists an unbounded sequence of solutions $(t_n, u_{1,n}, \dots, u_{m,n})$ to the following
\[
\begin{cases}
\begin{aligned}
&(-\Delta)^{\alpha_j} u_{j,n}=(t_n^{\theta_j}+\abs{u_{j+1,n}})^{p_j},\, j =1, \dots, m-1  \\
&(-\Delta)^{\alpha_m} u_{m,n}=(t_n^{\theta_m}+\abs{u_{1,n}})^{p_m}  
\end{aligned} & \text{ in } B_1 \subset \mathbb{R}^N, \\
\frac{\partial^k u_{j,n}}{\partial \nu^k}=0, \, k=0, \dots, \alpha_j-1, \, j =1, \dots, m  &  \text{ on } \partial B_1,
\end{cases}
\]
where $N > 2 \max \{ \alpha_j \}_j $ and $\theta_j$ are such that 
\begin{equation}\label{theta} \theta_j p_j > \theta_{j-1}, \, \forall j. \end{equation}
For instance, one can call 
\[ a_j=1+j (\prod_{k=1}^m p_k - 1) \]
and choose 
\[ \theta_j=\frac{a_{j-1}}{\prod_{k=2}^j p_k} \]
for $j=2, \dots, m$, and $\theta_1=1$. 
The proof of this step relies on a fixed point lemma due to Azizieh and Cl\'ement \cite[Lemma A.2]{AziziehClement02}, see \cite[Proposition 1]{Schiera18} for the case $m=2$.

\medskip

\noindent \textit{Step 2}. Blow up analysis. We can assume without loss of generality that 
\[ \frac{t_n^{\theta_j}}{\norm{u_{j,n}}_{\infty}} \to 0, \, j=1,\dots, m, \] 
as follows by choosing 
\[ \tilde u_{j,n}=\frac{u_{j,n}}{t_n^{\theta_{j-1}}}, \, \, \lambda_{j,n}=t_n^{\theta_jp_j - \theta_{j-1}} \]
and applying the comparison principle. Here we exploit \eqref{theta} to have $\lambda_n \to \infty$.
Moreover, assume that the maximum of $u_{k,n}$ is attained in $0$ for any $k$. 
We define 
\[ \hat u_{j,n} (y)=\frac{u_{j,n} (C_n^{-1} y)}{A_{j,n}}, \]
where 
\[ A_{j,n}=C_n^{\sigma_j}, \]
\[ C_n=\sum_{j} \norm{u_{j,n}}_{\infty}^{1/\sigma_j} \]
and moreover 
\[ \sigma_1=\frac{2\sum_{k=1}^{m} \alpha_k \prod_{j=1}^{k-1} p_j}{\prod_{j=1}^m p_j -1}, \, \sigma_j=-2\alpha_j+p_j\sigma_{j+1}. \]
This by a limit procedure gives a nontrivial solution to 
\begin{equation}\label{sysRN}
 \begin{cases}
(-\Delta)^{\alpha_j} u_{j,n}=\abs{u_{j+1,n}}^{p_j},\, j =1, \dots, m-1 \, &\text{ on } \mathbb{R}^N \\
(-\Delta)^{\alpha_m} u_{m,n}=\abs{u_{1,n}}^{p_m}  \, &\text{ on } \mathbb{R}^N
\end{cases} \end{equation}
due to our choice of the parameters $A_{j,n}$ and $\sigma_j$. 
This limit solution is nontrivial since 
\[ \sum_{i \ne k} \norm{u_{i,n}}_{\infty}^{1/\sigma_i} \le \norm{u_{k,n}}_{\infty}^{1/\sigma_k}(m-1) \]
for at least one value $k$. Indeed, if not, then upon summation
\[ (m-1)\sum \norm{u_{i,n}}_{\infty}^{1/\sigma_i} > (m-1)\sum \norm{u_{i,n}}_{\infty}^{1/\sigma_i}, \]
a contradiction. 
Assume for instance that $k=1$ and call 
\[ b_n=\frac{\sum_{i \ne 1} \norm{u_{i,n}}_{\infty}^{1/\sigma_i}}{\norm{u_{1,n}}_{\infty}^{1/\sigma_1}} \le m-1. \]
Then, 
\[ (\hat u_{1,n})^{1/\sigma_1}(0)=\frac{\norm{u_{1,n}}_{\infty}^{1/\sigma_1}}{\sum\norm{u_{i,n}}^{1/\sigma_i}}= \frac{1}{1+b_n} \ge \frac{1}{m}, \]
and in particular the limit is nontrivial. 

\medskip

\noindent \textit{Step 3}. We prove that the maximum of $u_{k,n}$ is attained in $0$ for any $k$, as the following Lemma shows.
\begin{lemma}\label{lemma2}
Let $(u_1, \dots, u_m)$ be a nontrivial solution to 
\[
\begin{cases}
\begin{aligned}
&(-\Delta)^{\alpha_j} u_{j,n}=f_j(u_{j+1}),\, j =1, \dots, m-1  \\
&(-\Delta)^{\alpha_m} u_{m,n}=f_m(u_1) 
\end{aligned} & \text{ in } B_1 \subset \mathbb{R}^N,  \\
\frac{\partial^k u_{j,n}}{\partial \nu^k}=0, \, k=0, \dots, \alpha_j-1, \, j =1, \dots, m  &  \text{ on }\partial B_1,
\end{cases}
\]
where $N > 2 \max \{ \alpha_j \}_j$ and $f_j : [0, \infty) \to \mathbb{R}$ are continuous, positive and non decreasing. 
Then $u_1, \dots, u_m$ are radially symmetric and strictly decreasing in the radial variable. 
\end{lemma}
\noindent The proof is analogous to that of \cite[Proposition 3]{Schiera18}.

\medskip

\noindent \textit{Step 4}. Finally, we notice that \cite[Theorem 19.1]{MitidieriPohozaev01} can be extended easily to the case of $m>2$ equations as follows. 
\begin{theorem}
Let $p_j >1$, $\alpha_j \in \mathbb{N}$, $j=1,\dots, m$, and assume that there exists $l \in \{ 1, \dots, m \}$ such that
\[
N+2 \sum_{k=1}^m \alpha_{k+l} \prod_{j=0}^{k-1} p_{j+l} - N \prod_{j=1}^m p_j \ge 0, 
\]
where we impose $p_{k+m}=p_k$ and $\alpha_{k+m}=\alpha_k$ for any $k=1, \dots, m$.  Assume further that $(u_1, \dots, u_m)$ is a weak solution to 
\eqref{sysRN}. 
Then $u_j=0$ for any $j=1,\dots, m$. 
\end{theorem}

%--------------------------------------------------------------------------------------------------------------------------
\subsection{Uniqueness}%-----------------------------------------------------------------------------------------
%--------------------------------------------------------------------------------------------------------------------------
\noindent We first give the proof in the case $m=2$ and then we proceed inductively. System \eqref{sysDir} reads as follows
\[ \begin{cases}
\begin{aligned}
(-\Delta)^{\alpha} u=\abs{v}^q \\
(-\Delta)^{\beta} v= \abs{u}^p
\end{aligned} &\text{ in } B_1,  \\
\frac{\partial^{r} u}{\partial \nu^{r}}=0, \, r=0, \dots, \alpha-1, & \text{ on }\partial B_1, \\
\frac{\partial^{r} v}{\partial \nu^{r}}=0, \, r=0, \dots, \beta-1, & \text{ on }\partial B_1.
\end{cases} \]

\noindent Assume without loss of generality that $\alpha \le \beta$.
We take two nontrivial solutions $(u, v)$ and $(w, z)$, and the parametrization
\[ \tilde w(r)=\lambda^s w(\lambda r), \, \tilde z(r)=\lambda^t z(\lambda r) \]
where $t=\frac{2\alpha p + 2\beta}{pq-1}$, $s=\frac{2\beta q + 2\alpha}{pq-1}$. Notice that $s, t$ are well defined if $pq \ne 1$.
Moreover we build the same table as in the previous sections with columns
\[ u-\tilde w, \, \Delta(u-\tilde w), \dots, \Delta^{\alpha-1} (u-\tilde w), \, v-\tilde z, \dots, \Delta^{\beta -1} (v-\tilde z) \]
if $\alpha$ is even, whereas
\[ u-\tilde w, \, \Delta(u-\tilde w), \dots, \Delta^{\alpha-1} (u-\tilde w), \, -v+\tilde z, \dots, \Delta^{\beta -1} (-v+\tilde z)\]
if $\alpha$ is odd.

\noindent Assume that (for even $\alpha$, and similarly for odd $\alpha$)
\begin{equation}\label{initialconf}
\begin{split}
(u- \tilde w)(0)=0, \, (-\Delta)^{k}(u- \tilde w)(0)<0, \, k=1, \dots, \alpha-1, \\
(-\Delta)^{k}(v- \tilde z)(0)<0, \, k=0, \dots, \beta-1
\end{split}
\end{equation}
is the initial configuration of the columns. We obtain a sequence $\{ R_j \}$ as in \autoref{sec:eqn} and assume that this is finite.

\noindent Let $n$ be the number of zeros of the first column. Then by \autoref{lem:zeros}, the $\alpha$-th column  has at least $n+1$ zeros, and as a consequence the next one must have $n$, $n+1$ or more zeros. Knowing that the $(\alpha+\beta)$-th column has $n$ or $n-1$ zeros, one has (again by \autoref{lem:zeros}, see \autoref{rmk:zeros}) that the $(\alpha +1)$-th column cannot have strictly more than $n$ zeros. Hence, it has $n$ zeros. However, $(v-\tilde z)(s)$ has opposite sign with respect to $(u- \tilde w)(s)$ in $(0, R_1)$, hence they have opposite sign in the last interval as well. Therefore, $(u-\tilde w)(\min \{1, 1/\lambda \})>0$ implies $\lambda >1$, whereas $(v-\tilde z)(\min \{1, 1/\lambda \})<0$ gives $\lambda <1$, a contradiction.

\noindent If $\alpha=1$, then as above we prove that the column $\tilde z- v$ cannot have strictly more than $n$ zeros. However, it must have at least $n$ zeros, as $u(0)=\tilde w(0)$, thus exactly $n$ zeros. Again, we have a contradiction.

\noindent Let us assume that for another initial configuration $\mathcal{A}$ we do not reach a contradiction as above. In $(0, R_1)$ the signs of the columns from the second to the last one are the same as in $\mathcal{A}$, and the first column must have the same sign as the second one, due to the maximum principle and the assumption $u(0)=\tilde w(0)$. Let us call $\mathcal{A}_1$ the configuration in  $(0, R_1)$, given $\mathcal{A}$ in $0$.
It turns out that one can reach the configuration $\mathcal{A}_1$ starting from \eqref{initialconf}.
Indeed, given \eqref{initialconf}, all the columns from the second to the second-to-last can be $=0$ in $R_1$. Then, it is sufficient to impose $=0$ in $R_1$ the columns which have different signs with respect to $\mathcal{A}_1$. If the first column has different sign, then it is enough to note that, once the second column has changed sign, the first column can be $=0$ and change sign as well. Analogously, one can change the sign of the last column once the first one has been $=0$. See \autoref{tableconf} for an example.

\begin{table}\centering\caption{Passing from \eqref{initialconf} to the configuration $u-\tilde w<0$, $\Delta(u-\tilde w)<0$, $(v- \tilde z)>0$, $\Delta (v- \tilde z)>0$, $\Delta^2(v-\tilde z)>0$. }
\label{tableconf}
\begin{adjustbox}{max width=\textwidth}
\begin{tabular}{r|cccccc}
&$(u-\tilde w)(s)$&$\Delta( u- \tilde w)(s)$&$ (v- \tilde z)(s)$&$ \Delta (v- \tilde z)(s)$&$\Delta^2(v-\tilde z)(s)$\\ \hline
$s=0$&=0&>0&<0&>0&<0\\ \hline
$s \in (0, R_1)$&>0&>0&<0&>0&<0\\ \hline
$s=R_1$&>0&=0&=0&>0&<0\\ \hline
$s \in (R_1, R_2)$&>0&<0&>0&>0&<0\\ \hline
$s=R_2$&=0&<0&>0&>0&<0\\ \hline
$s \in (R_2, R_3)$&<0&<0&>0&>0&<0\\ \hline
$s=R_3$&<0&<0&>0&>0&=0\\ \hline
$s \in (R_3, R_4)$&<0&<0&>0&>0&>0\\ \hline
$\vdots$ &$\vdots$ &$\vdots$ &$\vdots$ &$\vdots$ &$\vdots$ 
\end{tabular}
\end{adjustbox}
\end{table}

Therefore, if from any other initial configuration $\mathcal{A}$ we do not have a contradiction, then this would be possible given \eqref{initialconf} as well.

\noindent We have thus proved that the sequence $\{R_j\}$ has to be infinite. However, in this case we reach a contradiction as in the previous sections, as we apply Gronwall with
\[ U(r)=(u(r), -\Delta u(r), \dots, (-\Delta)^{\alpha-1} u(r), v(r), \dots, (-\Delta)^{\beta-1} v(r))\]
for $ 0 \le r \le 1$ and
\[  W(r)=(\tilde{w}(r), -\Delta \tilde w(r), \dots,  (-\Delta)^{\alpha-1} \tilde w(r), \tilde z(r), \dots (-\Delta)^{\beta-1} \tilde z(r)) \]
for $0 \le r \le 1/\lambda$ and
$$F(x_1, x_2, \dots, x_{\alpha}, y_1, \dots, y_{\beta})=(x_2, x_3, \dots, x_{\alpha-1}, y_1^q, y_2, \dots, y_{\beta -1}, x_1^p)\ .$$

\noindent This proves that in $0$ all the columns are zero. Therefore, again by Gronwall's Lemma, we have $u=\tilde w$ and $v=\tilde z$, which in turn gives $(u,v )=(w, z)$.

\noindent The proof in the case $m>2$ follows by induction, once we parametrize a second solution $(w_1, \dots, w_m)$ as follows
\[ \tilde w_i(r)=\lambda^{s_i} w(\lambda r), \, i=1, \dots, m, \]
where $\lambda$ is chosen such that $\tilde w_1(0)=u_1(0)$, whereas
\[ s_1=\frac{2\sum_{j=1}^{m} \alpha_j \prod_{k=1}^{j-1} p_k}{\prod_{k=1}^m p_k-1}   \]
and
\[ s_{i+1}=\frac{s_i+2\alpha_i}{p_i}, \, i=1, \dots, m-1.  \]
Assuming as induction hypothesis that the last column corresponding to the first $m$ equations can not have less zeros than the first one, and taking $m=1$ as the base case (see \autoref{lem:zeros}), then one proves that that property holds for $m+1$ as well, by the same arguments as above.
More precisely, the induction hypothesis implies that the last column corresponding to the first $m$ equations must have at least one zero more than the first one. By exploiting \autoref{rmk:zeros}, and knowing that the last column has at most $n$ zeros, one proves as above that $u_1-\tilde w_1$ and $u_{\alpha}-\tilde w_{\alpha}$ must have opposite signs at the last step, which gives the contradiction. As for the case $\{ R_j \}$ infinite, the contradiction follows by applying Gronwall's lemma.

\noindent The proof of Theorem \ref{teo:sysDir} is now complete.

\begin{remark}
Notice that the restriction $\alpha_j \le 4$ is necessary as we need to exploit \autoref{lem:zeros}. Actually, if we could extend \autoref{lem:zeros} to higher order operators, then it would be possible to extend \autoref{teo:sysDir} to more general operators as well.
\end{remark}

\subsection{Some natural boundary conditions: proof of Corollary \ref{teo:Var}}
\noindent Notice that \eqref{sys:Var} can be written as a system of $\sum \lceil \alpha_j/2 \rceil$ equations with Dirichlet boundary conditions. Let for instance $\alpha$ be even, and set $u_k=\Delta^{2k} u$. Then
\[ \begin{cases}
(-\Delta)^{\alpha} u=\abs{u}^p, \, & \text{ in } B_1\\
\Delta^{2k} u=0, \, 2k \le \alpha-1, & \text{ on }\partial B_1 \\
\frac{\partial}{\partial \nu} \Delta^{2k} u= 0, \,  2k+1 \le \alpha-1, & \text{ on }\partial B_1
\end{cases} \]
reads as
\[ \begin{cases}
\begin{aligned}
&\Delta^2 u_j=\abs{u_{j+1}},\, j =1, \dots, \alpha/2-1,  \\
&\Delta^2 u_{\alpha/2}=\abs{u_{1}}^{p},
\end{aligned} & \text{ in }B_1\\
u_j=\frac{\partial u_j}{\partial \nu}=0, \, j =1, \dots, \alpha/2 &\text{ on } \partial B_1,
\end{cases} \]
which is a particular case of \eqref{sysDir}.

\noindent Let us point out that the boundary conditions in \eqref{sys:Var} satisfy the complementing condition \cite{AgmonDouglisNirenberg59}, which here read as follows
\begin{definition}\label{def:ADN}
We say that the complementing condition holds for
\[ \begin{cases}
(-\Delta)^{\alpha} u=\abs{u}^p, \, & \text{ in }B_1\\
B_j(x, D) u=h_j, \, \text{ for } j=1, \dots, \alpha, & \text{ on }\partial B_1
\end{cases} \]
if, for any nontrivial tangent vector $\tau (x)$, the polynomials in $t$ $B_j'(x; \tau + t \nu)$ are linearly independent modulo the polynomial $(t-i \abs{\tau})^{\alpha}$, where $B_j'$ represents the highest order part of $B_j$.
\end{definition}
\noindent Consider the particular case $\alpha=4$ and let $\abs{\tau}=1$. Then $B_1'(x, \tau + t \nu)=1$, $B_2'(x, \tau + t \nu)=t$, $B_3'(x, \tau + t \nu)=t^4+1$ and $B_4'(x, \tau + t \nu)=t^5+t$. Dividing these polynomials by $(t- i)^4$, we get $1$, $t$, $4it^3+6t^2-4it$ and $-10t^3+20it^2+16t-4i$ as remainders, which are linearly independent.
The general case follows from the system \eqref{sysDir}. Indeed, one can extend \autoref{def:ADN} to the case of systems and prove that a system of $m$ equations with Dirichlet boundary conditions satisfy this extended condition, see \cite{AgmonDouglisNirenberg64}.
%--------------------------------------------------------------------------------------------------------------------------
\subsection{Navier's boundary conditions: proof of Corollary \ref{teo:Navier}}%----------------------------------------------------------------------------------------------------
%--------------------------------------------------------------------------------------------------------------------------

Recall that, given a nontrivial solution to \eqref{sys:Nav}, then it is positive, radially symmetric and strictly decreasing in the radial variable, see \cite[Theorem 7.3]{GazzolaGrunauSweers10}. This reduces the problem to system \eqref{sysDir} and thus  \autoref{teo:Navier} follows from \autoref{teo:sysDir}.
Let for instance $m=1$. Then
\[ \begin{cases}
(-\Delta)^{\alpha} u=\abs{u}^p, \, &  \text{ in } B_1\\
\Delta^k u=0, \, k \le \alpha-1, & \text{ on } \partial B_1
\end{cases} \]
becomes
\[ \begin{cases}
\begin{aligned}
&-\Delta u_j=\abs{u_{j+1}},\, j =1, \dots, \alpha-1,  \\
&-\Delta u_{\alpha}=\abs{u_{1}}^{p},
\end{aligned} & \text{ in } B_1 \\
u_j=0, \, j =1, \dots, \alpha &  \text{ on } \partial B_1,
\end{cases} \]
where $u_j=\Delta^j u$.

%%%%%%%%%%%%%%%%%%%%%%%%%%%%%%%%%%%%%%%%%%%%%%%%%%%%%%

\end{document}